\documentclass[reqno]{amsart}

\usepackage{amscd,amssymb,amsmath,graphicx,verbatim}
\usepackage[all]{xy}
\newtheorem {thm}{Theorem}[section]
\newtheorem{lem}[thm]{Lemma}
\newtheorem{prop}[thm]{Proposition}
\newtheorem{cor}[thm]{Corollary}
\newtheorem{df}[thm]{Definition}

\newtheorem{ex}[thm]{Example}
\newtheorem{exs}[thm]{Examples}
\newtheorem{rem}[thm]{Remark}

\begin{document}

\title{Reversible ring property via idempotent elements}

\author[H. Kose]{Handan Kose}
\address{Handan Kose, Department of Mathematics, Ahi Evran University, Kirsehir, Turkey}
\email{handan.kose@ahievran.edu.tr}

\author[B. Ungor]{Burcu Ungor}
\address{Burcu Ungor, Department of Mathematics, Ankara
University, Ankara, Turkey}\email{bungor@science.ankara.edu.tr}

\author[A. Harmanci]{Abdullah Harmanci}
\address{Abdullah Harmanci, Department of Mathematics, Hacettepe
University, Ankara,~ Turkey}\email{harmanci@hacettepe.edu.tr}

\date{}

\begin{abstract}

Regarding the question of how idempotent elements affect
reversible property of rings, we study a version of reversibility
depending on idempotents. In this perspective, we introduce {\it
right} (resp., {\it left}) {\it $e$-reversible rings}. We show
that this concept is not left-right symmetric. Basic properties of
right $e$-reversibility in a ring are provided. Among others it is
proved that if $R$ is a semiprime ring, then $R$ is right
$e$-reversible if and only if it is right $e$-reduced if and only
if it is $e$-symmetric if and only if it is right
$e$-semicommutative. Also, for a right $e$-reversible ring $R$,
$R$ is a prime ring if and only if it is a domain. It is shown
that the class of right $e$-reversible rings is strictly between
that of $e$-symmetric rings and right $e$-semicommutative rings.
\end{abstract}
\maketitle
\section{Introduction}
Throughout this paper, all rings are associative with identity.
For a ring $R$, let $N(R)$, Id$(R)$ and $C(R)$ denote the set of
all nilpotents, the set of all idempotents and the center of $R$,
respectively. We denote the $n\times n$ full (resp., upper
triangular) matrix ring over $R$ by $M_n(R)$ (resp., $U_n(R))$,
and $D_n(R)$ stands for the subring of $U_n(R)$ having all
diagonal entries are equal and $V_n(R) = \{(a_{ij})\in D_n(R)\mid
a_{ij} = a_{(i+1)(j+1)}$ for $i = 1,\dots,n - 2$ and $j =
2,\dots,n-1\}$ a subring of  $D_n(R)$ and $E_{ij}$ denote the
matrix unit in $M_n(R)$ whose $(i, j)$-th entry is 1 and the
others are zero. The rings $\Bbb Z$ and $\Bbb Z_n$ denote the ring
of integers and the ring of integers modulo $n$ where $n$ is an
integer with $n\geq 2$.

A ring is usually called {\it reduced} if it has no nonzero
nilpotents. In \cite{Co}, a ring $R$ is called {\it reversible} if
$ab = 0$ implies $ba = 0$ for $a$, $b\in R$. In \cite{AC} it is
used the term $ZC_2$ for the reversible property. Lambek in
\cite{La} calls a ring $R$ {\it symmetric} if, for all $a$, $b$,
$c\in R$, $abc = 0$ implies $acb = 0$. $ZC_3$ is used for
symmetricity for rings in \cite{AC}. Reduced rings are both
reversible and symmetric in \cite{MW}. $e$-symmetric rings and
right(left) $e$-reduced rings are introduced  as a generalization
of symmetric rings and reduced rings respectively. Let $R$ be a
ring and $e\in$ Id$(R)$. The ring $R$ is called {\it
$e$-symmetric} if $abc = 0$ implies $acbe = 0$ for all $a$, $b$,
$c\in R$, and also $R$ is called {\it right} (resp., {\it left})
{\it $e$-reduced} if $N(R)e = 0$ (resp., $eN(R) = 0)$. It is
proved that right $e$-reduced rings are $e$-symmetric. An
idempotent $e$ of a ring $R$ is called {\it left} (resp., {\it
right}) {\it semicentral} if $ae = eae$ (resp., $ea = eae)$ for
each $a\in R$. In \cite{KUH}, a ring $R$ is called {\it right}
(resp. {\it left}) {\it $e$-semicommutative} if for any $a$, $b\in
R$, $ab = 0$ implies $aRbe = 0$ (resp. $eaRb = 0$), $R$ is {\it
$e$-semicommutative} in case $R$ is both right and left
$e$-semicommutative. Every commutative ring and every right
$e$-reduced ring is a right $e$-semicommutative ring.

\section{Right $e$-reversible rings }

In this section we deal with a generalization of an $e$-symmetric ring and that of a reduced ring,
which shall be said to be $e$-reversible. $e$-reversible context also generalizes $e$-reduced concept in \cite{MW}. We show that reduced rings
are $e$-reversible, and next that the class of $e$-reversible rings is quite large. For a ring $R$ we investigate the idempotents of $D_n(R)$ and some subrings of $M_n(R)$ where $n\geq 2$, and apply this to observe the relation between the $e$-reversibility of $R$ and the $E$-reversible ring property of $D_3(R)$ and some subrings of $M_n(R)$.

\begin{df}{\rm Let $R$ be a ring and $e\in$ Id$(R)$ with $e\neq 0$. Then $R$ is called {\it right
$e$-reversible} (resp., {\it left $e$-reversible)} if for any $a$,
$b\in R$, $ab = 0$ implies $bae = 0$ (resp., $eba = 0$). The ring
$R$ is {\it $e$-reversible} if it is both left and right
$e$-reversible.}
\end{df}

In the next result, we give a characterization of right
$e$-reversibility in terms of subsets of rings and its proof is
straightforward.

\begin{prop} The following are equivalent for a ring $R$.
\begin{enumerate}
    \item $R$ is right $e$-reversible.
    \item For any nonempty subsets $A, B$ of $R$, being $AB=0$
    implies $BAe=0$.
\end{enumerate}
\end{prop}

Clearly, a ring $R$ is reversible if and only if $R$ is
$1$-reversible. The following example shows that $e$-reversibility
is not left-right symmetric. Also, the reversibility of a ring
with respect to an idempotent $e$ depends on $e$. There are rings
$R$ and idempotents $e_1$ and $e_2$ such that $R$ is right
$e_1$-reversible but not right $e_2$-reversible as the following
example shows.

\begin{ex}\label{ikie}{\rm Consider the ring $R = U_2(\Bbb Z)$ and
$E_1 = \begin{bmatrix}1&1\\0&0 \end{bmatrix}$, $E_2 =
\begin{bmatrix}0&0\\0&1
\end{bmatrix}\in $ Id$(R)$. Then the following hold. \begin{enumerate}\item $R$ is not reversible.\item  $R$ is right $E_1$-reversible but not left  $E_1$-reversible.\item $R$ is left $E_2$-reversible but not right  $E_2$-reversible.
\end{enumerate}}
\end{ex}
\begin{proof} (1) $R$ is not reversible, in fact, $E_2E_1 = 0$ but $E_1E_2\neq
0$.\\
Let $A$, $B\in R$ with $AB = 0$. Then $BA$ is of the form
$\begin{bmatrix}0&x\\0&0 \end{bmatrix}$ where $x\in \Bbb Z$. Assume that $x\neq 0$.\\
(2) On the one hand, $BAE_1 = 0$. Hence $R$ is right
$E_1$-reversible. On the other hand, $E_1BA \neq 0$. Thus  $R$ is
not left $E_1$-reversible.
\\ (3) $E_2BA =0$ implies that $R$ is left $E_2$-reversible, and $R$
is not right $E_2$-reversible since $BAE_2\neq 0$.
\end{proof}

For a reduced ring $R$, we now show that $M_n(R)$ is neither right
$e$-reversible nor left $e$-reversible for some $e\in $
Id$(M_n(R))$.

\begin{ex}\label{reduced}{\rm Let $R$ be a reduced ring and $E_{ij}$ denote the matrix unit in $M_n(R)$ whose $(i, j)$-th
entry is 1 and the others are zero. Consider $A = E_{23}$, $B =
E_{12}$ and $E = E_{11} + E_{33}\in M_n(R)$. Then $AB = 0$ and
$BA\neq 0$. Also, $BAE\neq 0$. Hence $M_n(R)$ is not right
$E$-reversible. Similarly, $EBA\neq 0$. Thus $M_n(R)$ is not left
$E$-reversible. Therefore $M_n(R)$ is neither right $E$-reversible
nor left $E$-reversible. }
\end{ex}

The following examples provide another sources of examples for
right $e$-reversible rings.

\begin{exs}\label{exs}{\rm (1) Every reversible ring is $e$-reversible.\\
(2) Every right $e$-reduced ring is  right $e$-reversible.\\
(3) Every  $e$-symmetric ring is  right $e$-reversible.}
\end{exs}
\begin{proof} (1) Let $R$ be a reversible ring with $e^2 = e\in R$. Assume that $ab = 0$.
Then $ba = 0$ and so $bae = 0$ and $eba=0$.\\
(2) Let $R$ be a right $e$-reduced ring and $a$, $b\in R$ with $ab
= 0$. Then $ba$ is nilpotent. Being $R$ right $e$-reduced, $bae =
0$.\\
(3) It is clear since the rings have identity.
\end{proof}

There are right $e$-reversible rings $R$ for some $e\in$ Id$(R)$
but not reversible as shown below. This yields that the converse
of Examples \ref{exs}(1) need not be true in general.

\begin{ex}\label{counter} {\rm For a reversible ring $R$, the ring $U_2(R)$ is right $E$-reversible for some $E\in$ Id$(U_2(R))$ but not
reversible.}
\end{ex}
\begin{proof} Let $A = \begin{bmatrix}0 & 1 \\0 & 1 \\\end{bmatrix}$, $B=\begin{bmatrix}1 & 0 \\0 & 0 \\\end{bmatrix}\in U_2(R)$,
we get $AB=0$ but $BA\neq 0$, hence $U_2(R)$ is not reversible.
Let $A=\begin{bmatrix}a & b \\  0 & c \\\end{bmatrix},
B=\begin{bmatrix}x & y \\0 & z \\\end{bmatrix}\in U_2(R)$
such that $AB=0$. Then we have $BA=\begin{bmatrix} xa & * \\0 & zc \\
\end{bmatrix}$. The reversibility of $R$ yields  $xa=0$ and $zc=0$. Let $E = \begin{bmatrix}1 & 0 \\0 & 0\end{bmatrix}$.
We get $BAE = 0$. Thus $R$ is right $E$-reversible.
\end{proof}

\begin{thm}\label{thm} Let $R$ be right $e$-reversible and right $1 - e$-reversible ring. Then $R$ is semiprime if and only if $R$ is reduced if and only if $R$ is reversible.
\end{thm}
\begin{proof} "Only if" part. Assume that $R$ is semiprime right $e$-reversible and right $1 - e$-reversible ring. Let $a^n = 0$ and $b = a^{n-1}$. Then $b^2 = 0$. Right $e$-reversibility of $R$ yields $brbe = 0$ for each $r\in R$. Since $r$ is arbitrary in $R$, we may replace $r$ by $er$ to get $berbe = 0$ for each $r\in R$. Being $R$ semiprime yields $be = 0$. Hence by induction $ae = 0$. By replacing $e$ by $1 - e$ in this proof we may have that $a(1 - e) = 0$. Hence $a = 0$. Thus $R$ is reduced.
\end{proof}

Let $R$ be a ring with a ring homomorphism $\sigma$. Consider the
ring $U_2(R)_{\sigma}$ consisting of all elements of $U_2(R)$ with
usual matrix addition in $U_2(R)$ and multiplication defined by
\begin{center} $\begin{bmatrix}a & b \\0 & c \end{bmatrix}\begin{bmatrix}x & y \\0 & z \end{bmatrix} =
\begin{bmatrix}ax & ay + b\sigma(z)\\0&cz\end{bmatrix}$ where
$\begin{bmatrix}a & b \\0 & c \\\end{bmatrix}$, $\begin{bmatrix}x
& y \\0 & z \\\end{bmatrix}\in U_2(R)_{\sigma}$.\end{center}

\begin{ex} {\rm Let $R$ be a ring, $\sigma$ a ring homomorphism of $R$ and $e\in$ Id$(R)$.
If $U_2(R)_{\sigma}$ is right $eI_2$-reversible, then $R$ is right
$e$-reversible. The converse holds if $e\in $ Ker$\sigma$. }
\end{ex}



It is well known that every reversible ring is abelian, but this
is not the case for right $e$-reversibility, Example \ref{ikie}
illustrates this claim. However, according to the following
result, right $e$-reversibility implies abelianness of the corner
ring $eRe$ of a ring $R$. In the meantime, we give the main
characterization of right $e$-reversibility.

\begin{thm}\label{ere}  Let $R$ be a ring and $e\in$ Id$(R)$. Then $R$ is a right $e$-reversible ring
if and only if $eRe$ is a reversible ring and $e$ is left
semicentral.
\end{thm}\label{semi}
\begin{proof} For the necessity, assume that $R$ is a right $e$-reversible ring. Then $e(1 - e)x = 0$ for all $x\in R$.
By assumption $(1 - e)xe = 0$. So $xe = exe$. Hence $e$ is left
semicentral. Let $x, y\in R$ with $(exe)(eye) = 0$. By assumption,
$(eye)(exe) = 0$. Hence $eRe$ is reversible. For the sufficiency,
suppose that $eRe$ is a reversible ring and $e$ is left
semicentral. Let $a$, $b\in R$ such that $ab = 0$. The idempotent
$e$ being left semicentral and $abe = 0$ imply $(eae)(ebe) = 0$.
Reversibility of $eRe$ gives rise to $(ebe)(eae) = 0$. Again we
invoke being $e$ left semicentral to get $bae = 0$. Thus $R$ is
right $e$-reversible.
\end{proof}

Similar to Theorem \ref{ere}, we have the following result.
\begin{prop} Let R be a ring and $e\in$ Id$(R)$. Then $R$ is left $e$-reversible if and only if $e$ is right semicentral in $R$ and $eRe$ is reversible.
\end{prop}

\begin{prop} Every right $e$-reversible ring is right  $e$-semicommutative.
\end{prop}
\begin{proof} Let $a$, $b\in R$ with $ab = 0$. For any $r\in R$, by hypothesis,
$baer = 0$ and again applying hypothesis to $baer = 0$ we get $aerbe = 0$. By Theorem \ref{ere}, $e$ is
left semicentral. Hence $arbe = 0$ for all $r\in R$. Therefore $R$
is right $e$-semicommutative.
\end{proof}

In \cite[Corollary 4.3]{MW} it is proved that every right
$e$-reduced ring is $e$-symmetric. Then we have the picture:
{\small \begin{center} $
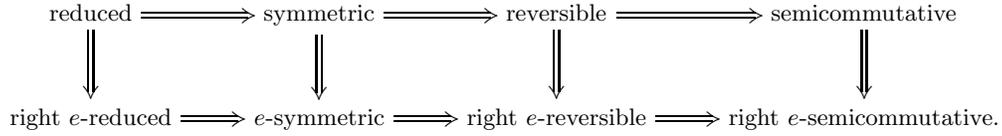
\hspace*{-0.3cm}\xymatrix{\mbox{reduced}
\ar@{=>}[r] \ar@{=>}[d] &\mbox{symmetric} \ar@{=>}[r] \ar@{=>}[d]&
\mbox{reversible}
\ar@{=>}[r] \ar@{=>}[d] &\mbox{semicommutative} \ar@{=>}[d]\\
\mbox{right}~e\mbox{-reduced} \ar@{=>}[r] &e\mbox{-symmetric}
\ar@{=>}[r] & \mbox{right}~e\mbox{-reversible} \ar@{=>}[r]
&\mbox{right}~e\mbox{-semicommutative}.\\} $
\end{center}}
There are right $e$-semicommutative rings but not right $e$-reversible.
It is known that there are semicommutative rings that are not
reversible.

\begin{ex}{\rm Let $R$ be a semicommutative ring that is not reversible.
Consider the ring H$_{(1,1)}(R)$ in \cite[Theorem 3.2]{KUH} and $E
= E_{11} + E_{21}$. It is proved that H$_{(1,1)}(R)$ is right
$E$-semicommutative. Let $a$, $b\in R$ with $ab = 0$ and $ba\neq
0$. The existence of $a$ and $b$ comes from $R$ not being
reversible. Let $A = aE$, $B = bE\in$ H$_{(1,1)}(R)$. Then $AB =
0$ but $BAE = BA\neq 0$. Hence H$_{(1,1)}(R)$ is right
$E$-semicommutative but not right $E$-reversible.}
\end{ex}

\begin{prop} Every right $e$-semicommutative reflexive ring is
right $e$-reversible.
\end{prop}
\begin{proof} Let $R$ be a right $e$-semicommutative reflexive
ring and $a,b\in R$ such that $ab=0$. Right $e$-semicommutativity
implies $aRbe=0$, and then reflexivity yields $beRa=0$. Since
$bea=0$, we have $bRea=0$, and so $bReae=0$. By \cite[Theorem
2.4.]{KUH}, $e$ is left semicentral. This entails $bRae=0$.
Therefore $bae=0$.
\end{proof}

\begin{rem}{\rm Clearly, every reversible ring is reflexive. But this is not the case for right
$e$-reversible rings. For example, the ring $R$ in Example
\ref{ikie} is right $E_1$-reversible but it is not reflexive, in
fact, for $a=\begin{bmatrix}0&1\\0&1
\end{bmatrix}, b=\begin{bmatrix}1&1\\0&0
\end{bmatrix}\in R$, $aRb=0$ but $bRa\neq 0$.}
\end{rem}

There are also right $e$-reversible rings but not $e$-symmetric.

\begin{ex}{\rm Let $R$ be a reversible ring that is not symmetric.
So $R$ is semicommutative. On the one hand,  by \cite[Examples
2.9(2)]{KUH}, $U_2(R)$ is not $E$-symmetric where
$E=\begin{bmatrix}1 & 0 \\0 & 0\end{bmatrix}$. On the other hand,
$U_2(R)$ is right $E$-reversible as in the proof of Example
\ref{counter}.}
\end{ex}

\begin{lem}\label{eae} Let $R$ be a right $e$-reversible ring and $a\in R$. Then the following hold.
\begin{enumerate}
\item If $ea = 0$, then $aRe = 0$.
\item If $ae = 0$, then $aRe = 0$.
\end{enumerate}
\end{lem}
\begin{proof} (1) Assume that $ea = 0$. For any $r\in R$, $ear = 0$.
By hypothesis, $are = 0$.\\(2) Suppose that $ae = 0$. For any
$r\in R$, $aer = 0$. By hypothesis, $erae = 0$. Since $e$ is left
semicentral, $rae = 0$. By right $e$-reversibility of $R$, we have
$aere = 0$. Again,  $e$ being left semicentral yields $are = 0$.
\end{proof}

\begin{cor}\label{hab1} For a ring $R$, consider the following
conditions:
\begin{enumerate}
\item $R$ is right $e$-reversible,
\item The right annihilator $r_R(eR)$ of $eR$ is contained in the left annihilator $l_R(Re)$ of
$Re$,
\item For any nonempty subset $A$ in $R$, $(eR)A = 0$ implies $A(Re) = 0$.
\end{enumerate}
Then {\rm(1)} $\Rightarrow$ {\rm(2)} $\Leftrightarrow$ {\rm(3)}.
\end{cor}
\begin{proof} (1) $\Rightarrow$ (2) Let $a\in r_R(eR)$. Then $eRa = 0$.
Since $R$ has an identity, $ea = 0$. By Lemma \ref{eae}(2), $aRe =
0$. It follows that $a\in l_R(Re)$. Thus $r_R(eR)\subseteq
l_R(Re)$.\\(2) $\Rightarrow$ (3)  $(eR)A = 0$ implies $A\subseteq
r_R(eR)$. Since $r_R(eR)\subseteq l_R(Re)$, $A\subseteq l_R(Re)$
and so $A(Re) = 0$.\\
(3) $\Rightarrow$ (2) Clear.
\end{proof}

We now give an example to show that there are rings in which
Corollary \ref{hab1} ({\rm(3)} $\Rightarrow$ {\rm(1)}) does not
hold.
\begin{ex}{\rm Let $R$ be a reduced ring and $S = M_n(R)$. Since $R$ is reduced,
it is reflexive. This yields by \cite[Theorem 2.6.(2)]{KwL} that
$S$ is a reflexive ring entailing that $(eS)A = 0$ implies $A(Se)
= 0$ for any $\emptyset \neq A\subseteq S$ and $e^2=e\in S$. On
the other hand, $S$ is not right $e$-reversible for some
idempotent $e\in S$ by Example \ref{reduced}.}
\end{ex}

Right idempotent reflexivity need not imply right
$e$-reversibility as shown below.

\begin{ex}\label{ucters}{\rm  Let $R$ denote the ring of \cite[Example 3.3]{KwL}. Let $S = F\langle a, b, c\rangle$
be the free algebra with non-commuting indeterminates $a$, $b$ and
$c$ over a field $F$ of characteristic zero and the ideal $I =$
$\langle aSb, a^2 - a\rangle$ and $R = S/I$. Consider $e = a\in$ Id$(R)$. Then
$ab = 0$ but $bae\neq 0$. Hence $R$ is not right $e$-reversible.
Let $h\in R$ and any $e^2 = e\in$ Id$(R)$, it is proved that $hRe
= 0$ implies $eRh = 0$, i.e., $R$ is right idempotent reflexive.
It also follows that for any subset $A$ of $R$, $ARe = 0$ implies
$eRA = 0$. }
\end{ex}

It is well known that every reversible ring is abelian, i.e., every idempotent is central.

\begin{cor} If $R$ is a right $e$-reversible ring and $f\in$ Id$(R)$, then the following hold.
\begin{enumerate}
\item $eRe$ is an abelian ring,
\item $afe = fae$ for any $a\in R$,
\item If $f\in eRe$, then $f$ is left semicentral in $R$.
\end{enumerate}
\end{cor}
\begin{proof} Since $R$ is right $e$-reversible, $e$ is left semicentral and $eRe$ is reversible by Theorem
\ref{ere}.\\
(1) $eRe$ being reversible implies that it is abelian.\\
(2)  Let $r\in R$. Since $e$ is left semicentral and $fe = efe$ is an idempotent in $eRe$, we have
$rfe = (ere)(efe)$. The abelian property of $eRe$ yields $(ere)(efe) = (efe)(ere)$.
So the left semicentrality of $e$ entails $rfe = (ere)(efe) = (efe)(ere) = fre$.\\
(3) Let $r\in R$. The ring $eRe$ being abelian and $e$ being left
semicentral imply $rf = rfe= erfe= (ere)(efe) = (efe)(ere)(efe) =
frf$.
 \end{proof}

\begin{prop}\label{bul} The following are equivalent for a ring $R$.
\begin{enumerate}
\item $R$ is right $e$-reversible.
\item For any $a$, $b\in R$, if $ab\in$ Id$(R)$, then $bae\in$ Id$(R)$ and $e$ is left semicentral.
\end{enumerate}
\end{prop}
\begin{proof} (1) $\Rightarrow$ (2) Let $R$ be a right $e$-reversible ring and $a$, $b\in R$ with $ab\in$ Id$(R)$.
By Theorem \ref{ere}, $e$ is left semicentral. Being $ab\in$ Id$(R)$ implies $a(1-ba)b = 0$. Then $(1-ba)bae = 0$ by (1).
Since $e$ is left semicentral, $bae = baebae$. So $bae\in$ Id$(R)$.\\
(2) $\Rightarrow$ (1) Suppose that for any $a$, $b\in R$ being
$ab\in$ Id$(R)$ implies $bae\in$ Id$(R)$. Let $a$, $b\in R$ with
$ab = 0$. Then $ab\in$ Id$(R)$ entailing $bae\in$ Id$(R)$. Hence
$bae = baebae = babae = 0$ by the facts that $ab = 0$ and $e$ is
left semicentral. Thus $bae = 0$. Therefore $R$ is right
$e$-reversible.
\end{proof}

\begin{cor} The following are equivalent for a ring $R$.
\begin{enumerate}
\item $R$ is right $e$-reversible.
\item For any $a$, $b\in R$, if $ab\in$ Id$(R)$, then $abe = bae$.
\end{enumerate}
\end{cor}
\begin{proof} (1) $\Rightarrow$ (2) Let $R$ be a right $e$-reversible ring and $a$, $b\in R$ with $ab\in$ Id$(R)$.
By Proposition \ref{bul}, $bae\in$ Id$(R)$. By Theorem \ref{ere},
we use the facts that $eRe$ is reversible and so abelian and $e$
is left semicentral to get \begin{center} $bae = baebae =
ebeeaeebeeae = ebe(eabe)eae = (eabe)(ebe)(eae) = (eabe)(ebae) =
(eae)(ebe)(ebae)= (eae)(ebae)(ebe) = eaeebaeebe = (abe)(abe) =
ababe = abe.$ \end{center} (2) $\Rightarrow$ (1) Let $a$, $b\in R$
with $ab=0$. By (2), $abe=bae$, and so $bae=0$. This completes the
proof.
\end{proof}

Let $R$ be a ring. Then $f\in$ Id$(R)$ is called {\it a left
minimal idempotent} if the left ideal $Rf$ is minimal. The set of
all left minimal idempotents of $R$ is denoted by ME$_l(R)$.
Recall that the ring  $R$ is called {\it left min-abel} if either
ME$_l(R)= \emptyset$ or every element of ME$_l(R)$ is left
semicentral. In \cite{LD}, a ring $R$ is {\it left (or right)
quasi-duo} if every maximal left (or right) ideal of $R$ is an
ideal, respectively and $R$ is MELT if every essential maximal
left ideal of $R$ is an ideal. Clearly, every left quasi-duo ring
is MELT.

\begin{thm}\label{ik} Let $R$ be a ring. Then the following are equivalent.
\begin{enumerate}
\item $R$ is a left min-abel ring.
\item $R$ is $e$-symmetric  for each $e\in$ ME$_l(R)$.
\item  $R$ is right $e$-reversible  for each $e\in$ ME$_l(R)$.
\end{enumerate}
\end{thm}
\begin{proof} (1) $\Leftrightarrow$ (2) Clear by \cite{MW}.\\
(2) $\Rightarrow$ (3) By Examples \ref{exs}.\\
(3) $\Rightarrow$ (1) Let $e\in$ ME$_l(R)$. Then $e$ is left
minimal idempotent. By Theorem \ref{ere}, $e$ is left semicentral.
We invoke the definition of left min-abel ring to obtain that $R$
is a left min-abel ring.
\end{proof}

In \cite{MW}, it is proved that $R$ is a left quasi-duo ring if
and only if $R$ is a left min-abel MELT ring. In this vein, we
state and prove the following.
\begin{thm}\label{ik1} Let $R$ be a ring. Then the following are equivalent.
\begin{enumerate}
\item $R$ is a left quasi-duo ring.
\item $R$ is a right $e$-reversible MELT ring for each $e\in ME_l(R)$.
\end{enumerate}
\end{thm}
\begin{proof} (1) $\Rightarrow$ (2) Let $R$ be a left quasi-duo ring. By definition, $R$ is MELT and $R$ is left min-abel by
\cite[Theorem 1.2]{Wei2}. Theorem \ref{ik} implies that $R$ is
right $e$-reversible.\\ (2) $\Rightarrow$ (1) Suppose that $R$ is
a right $e$-reversible MELT ring for each $e\in ME_l(R)$. By
Theorem \ref{ik}, $R$ is a left min-abel ring. Hence $R$ is a left
quasi-duo ring by \cite[Corollary 2.6]{MW}.
\end{proof}

In the next example, we illustrate that for any ring $R$ and its
idempotent $e$ and ideal $I$, the ring $R/I$ being right
$\overline{e}$-reversible need not imply $R$ being right
$e$-reversible, and then we investigate under which condition this
property satisfies.

\begin{ex}{\rm Consider the ring $R=U_2(\Bbb Z)$ and its idempotent $E=
\begin{bmatrix}0&0\\0&1
\end{bmatrix}$. By Example \ref{ikie}, $R$ is not right $E$-reversible. Consider the ideal $I=\begin{bmatrix}\Bbb Z&\Bbb Z\\0&0
\end{bmatrix}$ of $R$. Then $R/I=\left\{\begin{bmatrix}0&0\\0&a
\end{bmatrix}+I\mid  a\in \Bbb Z\right\}$. For any $A=\begin{bmatrix}0&0\\0&a
\end{bmatrix}+I, B=\begin{bmatrix}0&0\\0&b
\end{bmatrix}+I\in R/I$ with $AB=0$, we have $ab=0$ entailing that $a=0$ or $b=0$. It follows that $A=0$ or $B=0$.
Hence $BA\overline{E}=0$. Therefore $R/I$ is right
$\overline{E}$-reversible.}
\end{ex}

\begin{lem}\label{semicentral} Let $R$ be a ring, $e\in$ Id$(R)$ and $I$ an ideal of $R$ with $I$ reduced as a ring without
identity. If $R/I$ is right $\overline e$-reversible, then $e$ is
left semicentral.
\end{lem}
\begin{proof} Let $r\in R$. Since $\overline e \overline{(1-e)}\overline
r=0$ and $R/I$ is right $\overline e$-reversible, we have
$\overline{(1-e)} \overline r \overline e=0$. Hence $re-ere\in I$.
Since $(re-ere)^2=0$ and $I$ is reduced, $re=ere$. Therefore $e$
is left semicentral.
\end{proof}

\begin{thm}
Let $R$ be a ring with an ideal $I$ and $e\in$ Id$(R)$. If $R/I$
is right $\bar{e}$-reversible and $I$ is a reduced ring without
identity, then $R$ is right $e$-reversible.
\end{thm}
\begin{proof}
Let $a, b\in R$ with $ab=0$. Then $\bar{a}\bar{b}=0$ in $R/I$.
Since $R/I$ is right $\bar{e}$-reversible, $\bar{b}\bar{a}\bar{e}=
0$. So $bae\in I$. By  Lemma \ref{semicentral}, $e$ is left
semicentral in $R$. It gives rise to $(bae)^2=baebae=babae=0$. The
ideal $I$ being reduced implies $bae=0$. It follows that $R$ is
right $e$-reversible.
\end{proof}

It is known by \cite[Example 2.1]{KiL}, there exists a ring $R$
and its ideal $I$ such that $R$ is reversible, but $R/I$ is not
reversible. Since reversibility and $1$-reversibility are the
same, we deduce that the ring $R$ being right $e$-reversible need
not imply $R/I$ being right $\overline{e}$-reversible. In the next
result, we show when we have an affirmative answer.

\begin{prop} Let $R$ be an $e$-symmetric ring and $I$ an ideal of
$R$ with $I=r_R(J)$ for some subset $J$ of $R$. Then $R/I$ is
right $\overline{e}$-reversible.
\end{prop}
\begin{proof} Let $\overline{a}, \overline{b}\in R/I$ such that $\overline{a}\overline{b}=0$. It
follows that $Jab=0$. The $e$-symmetricity of $R$ entails
$Jbae=0$. Hence $bae\in I$. This implies
$\overline{b}\overline{a}\overline{e}=0$. Therefore $R/I$ is right
$e+I$-reversible.
\end{proof}

\begin{thm}\label{prod}  Let $(R_i)_{i\in I}$ be a family of rings for some index set $I$ and $e_i^2 = e_i\in R_i$ for each $i\in I$  and set $e = (e_i)\in \prod_{i\in I}R_i$. Then $R_i$ is right $e_i$-reversible for each $i\in I$ if and only if $\prod_{i\in I}R_i$ is right $e$-reversible.
\end{thm}
\begin{proof} Assume that $R_i$ is right $e_i$-reversible
for each $i\in I$. Let $a = (a_i)$, $b = (b_i)\in \prod_{i\in
I}R_i$ with $ab = 0$. Then $a_ib_i = 0$ for each $i\in I$. By
assumption $b_ia_ie_i = 0$ for each $i\in I$. Then $bae = 0$.
Conversely, suppose that  $\prod_{i\in I}R_i$ is right
$e$-reversible. Let $a_i$, $b_i\in R_i$ with $a_ib_i = 0$. Let $a
= (a_i)\in \prod_{i\in I}R_i$ with $i^{th}$ component is $a_i$,
elsewhere is 0 and  $b = (b_i)\in \prod_{i\in I}R_i$ with $i^{th}$
component is $b_i$, elsewhere is 0. Then $ab = 0$. By supposition
$bae = 0$. Hence $b_ia_ie_i = 0$. So $R_i$ is right
$e_i$-reversible for each $i\in I$.
\end{proof}

Recall that in \cite{MW}, a ring $R$ is called {\it right} (resp.
{\it left}) {\it e-reduced} if $N(R)e=0$ (resp. $eN(R)=0$), and
$R$ is called {\it $e$-symmetric} if $abc = 0$ implies $acbe = 0$
for all $a$, $b$, $c\in R$. By \cite[Corollary 4.3]{MW}, the
converse statement holds in case that $R$ is semiprime.

\begin{thm}\label{reg} Let $R$ be a semiprime ring. Then the following are equivalent.
\begin{enumerate}
\item $R$ is right $e$-reversible,
\item $R$ is right $e$-reduced,
\item $R$ is $e$-symmetric,
\item $R$ is right $e$-semicommutative.
\end{enumerate}
\end{thm}
\begin{proof} (1) $\Rightarrow$ (2) Let $a\in R$ with $a^n = 0$ for some integer $n$. We may assume that $n = 2k$
for some integer $k\geq 1$. Then $a^k(a^kr)=0$ for any $r\in R$.
By hypothesis, $(a^kr)a^ke = 0$. The ring $R$ being right
$e$-reversible implies that $(ra^ke)a^ke = 0$. Again using the
right $e$-reversibility of $R$ we get $(a^ke)r(a^ke) = 0$ for all
$r\in R$. We invoke here $R$ to be a semiprime ring, we get $a^ke
= 0$. By the left semicentrality of $e$, we have $(ae)^k = 0$.
If $k\neq 1$, we may assume that $k = 2l$ for some integer $l\geq 1$.
By a similar discussion, we reach $(ae)^l = 0$. Continuing in this way, we may have $ae = 0$. Therefore $R$ is right $e$-reduced.\\
(2) $\Rightarrow$ (1) Clear by Examples \ref{exs}. (2)
$\Leftrightarrow$ (3) $\Leftrightarrow$ (4)  is proved in
\cite[Proposition 2.10]{KUH}.
\end{proof}

Since every von Neumann regular ring is semiprime, we have the
following result as an immediate consequence of Theorem \ref{reg}.

\begin{cor}\label{spri} Let $R$ be a von Neumann regular ring. Then the following are equivalent.
\begin{enumerate}
\item $R$ is right $e$-reversible,
\item  $R$ is right $e$-reduced,
\item  $R$ is $e$-symmetric,
\item $R$ is right $e$-semicommutative.
\end{enumerate}
\end{cor}

\begin{prop}\label{pri} A ring $R$ is  right $e$-reversible and prime if and only if $R$ is a domain.
\end{prop}
\begin{proof} One way is clear. For the other way, let $R$ be a right $e$-reversible prime ring and $a$, $b\in R$ with $ab = 0$.
For any $r\in R$, $abr = 0$. By hypothesis, $brae =  0$. Then
$bRae=0$. Since $R$ is prime, $ae = 0$ or $b = 0$.  If $b=0$, then
there is nothing to do. So $ae=0$.  Multiplying $ae=0$ from the
right by $re$ for any $r\in R$ yields $aere = 0$. Since $e$ is
left semicentral, $aere = 0$ implies $are = 0$. It follows that $a
= 0$ since $R$ is prime and $e\neq 0$.
\end{proof}

\begin{cor} Let $R$ be a right $e$-reversible prime ring.
Then $R$ is directly finite.
\end{cor}
There are directly finite rings $R$ with an idempotent $e$ such
that $R$ is neither right $e$-reversible nor a prime ring as the
following example shows.
\begin{ex}{\rm Let $F$ be a field of characteristic not equal to $2$ and consider the ring $R=U_2(F)$. It is well known that $R$ is directly finite but
not prime.
Let $E=\left[\begin{array}{cc}   0 & 0 \\   0 & 1 \\ \end{array}
\right]\in $ Id$(R)$. For $A=\left[\begin{array}{cc}   0 & 1 \\
0 & 1 \\ \end{array} \right]$ and $B=\left[\begin{array}{cc}   1 & 1 \\
0 & 0 \\ \end{array} \right]\in R$, $AB=0$ but $BAE\neq 0$. Hence
$R$ is not right $E$-reversible.   }
\end{ex}

\section{Extensions of $e$-reversible Rings}
Let $R$ be an algebra over a commutative ring $S$. Due to Dorroh
\cite{Do}, consider the abelian group $R\oplus S$ with
multiplication defined by $(a, b)(c, d) = (ac + da + bc, bd)$
where $a$, $c\in R$, $b$, $d\in S$. By this operation $R\oplus S$
becomes a ring called {\it Dorroh extension of $R$ by $S$} and
denoted by $D(R, S)$. By definition, $S$ is isomorphic to a
subring of $R$. By this reason we may assume that $S$ is contained
in the center of $R$ and use this fact in the sequel.
\begin{lem} Let $(a, b)\in D(R, S)$. Then $(a, b)\in$ Id$(D(R, S))$  if and only if $a + b\in$ Id$(R)$, $b\in$ Id$(S)$.
\end{lem}
\begin{proof} Let $(a, b)\in$ Id$(D(R, S))$. Then $(a, b)^2 = (a^2 + 2ba, b^2) = (a, b)$ implies $b^2 = b$ and $a^2 + 2ba + b^2 = a + b$.
Conversely, assume that $a^2 + 2ba + b^2 = a + b$ and $b^2 = b$.
Then $a^2 + 2ba = a$. So $(a^2 + 2ba, b^2) = (a, b)$. Hence $(a,
b)^2 = (a, b)\in$ Id$(D(R, S))$.
\end{proof}
\begin{prop} Let $e\in$ Id$(R)$. Then $R$ is right $e$-reversible if and only if $D(R, S)$ is right $(e, 0)$-reversible.
\end{prop}
\begin{proof} For the necessity, assume that $R$ is right $e$-reversible.
Let $A = (a, b)$, $B = (c, d)\in D(R, S) $ with $AB = 0$. Then $ac
+ da + bc = 0$ and $bd = 0$. Hence $ac + da + bc + bd =  0$. Since
$S$ is contained in the center of $R$, $ac + da + bc + bd =  0$
implies $(a + b)(c + d) = 0$. By assumption, $(c + d)(a + b)e  =
0$. It follows that $(c, d)(a, b)(e, 0) = 0$. For the sufficiency,
suppose that $D(R, S)$ is right $(e, 0)$-reversible. Let $a$,
$b\in R$ with $ab = 0$. Then $(a, 0)(b, 0) = 0$. By supposition
$(b, 0)(a, 0)(e, 0) = 0$. It implies $(bae, 0) = 0$. Then $bae =
0$. Hence $R$ is right $e$-reversible.
\end{proof}

\cite[Example 2.1]{KiL} shows that a ring $R$ being reversible
need not imply $R[x]$ being reversible. Due to this fact, we can
say that the ring $R$ being right $e$-reversible need not imply
$R[x]$ being right $e$-reversible. We now deal with this property
for Armendariz rings. Recall that a ring $R$ is called {\it
Armendariz} if for any $f(x) = \sum^n_{i=0}a_ix^i$, $g(x) =
\sum^m_{j = 0}b_jx^j\in R[x]$, being $f(x)g(x) = 0$ implies
$a_ib_j = 0$ where $0\leq i\leq n$, $0\leq j\leq m$.
\begin{prop} Let $R$ be an Armendariz ring. Then the following are
equivalent.
\begin{enumerate}
    \item $R$ is a right  $e$-reversible ring for each $e\in$
    Id$(R)$.
    \item $R[x]$ is a right  $e$-reversible ring for each $e\in$
    Id$(R[x])$.
    \item $R[[x]]$ is a right  $e$-reversible ring for each $e\in$
    Id$(R[[x]])$.
\end{enumerate}
\end{prop}
\begin{proof} Note first that by \cite{KimL}, for any Armendariz ring $R$,
the idempotents in $R[x]$ and $R[[x]]$ belong to $R$ and $R$ is an abelian ring itself.\\
(1) $\Rightarrow$ (2) Assume that $R$ is a right $e$-reversible
ring for $e\in$ Id$(R)$. Let $f(x) = \sum^n_{i=0}a_ix^i$, $g(x) =
\sum^m_{j = 0}b_jx^j\in R[x]$ with $f(x)g(x) = 0$. By hypothesis,
$a_ib_j = 0$. The ring $R$ being right $e$-reversible yields
$b_ja_ie = 0$ for $0\leq i\leq n$, $0\leq j\leq m$.
This implies that $g(x)f(x)e = 0$. So $R[x]$ is right  $e$-reversible.\\
(2) $\Rightarrow$ (1) Clear. (1) $\Leftrightarrow$ (3) By a similar discussion in (1) $\Leftrightarrow$ (2).
\end{proof}

Let $R$ be a reduced ring. Then $D_2(R)$ is $e$-reversible for
each idempotent $e$. In fact, $D_2(R)$ and $V_n(R)$ are reversible
for each positive integer $n$, therefore they are $e$-reversible
for every idempotent $e$. The following example shows that the
assumption $R$ being reduced in proving $D_2(R)$ being reversible
is not superfluous.

\begin{ex} {\rm Let $R$ be a reduced ring and consider the ring $D_2(U_2(R))$.
Let \linebreak $A=\left[\begin{array}{cc}   \left[
\begin{array}{cc} 0 & 1\\   0 & 0 \\ \end{array} \right] & \left[
\begin{array}{rr} -1
& 1 \\   0 & -1 \\ \end{array} \right] \\
\left[ \begin{array}{cc}   0 & 0 \\   0 & 0 \\\end{array} \right]&
\left[\begin{array}{cc}  0 & 1 \\  0 & 0 \\ \end{array}\right]
\\ \end{array} \right]$ and $B=\left[\begin{array}{cc}   \left[ \begin{array}{cc}   0 &
1\\   0 & 0 \\ \end{array} \right] & \left[ \begin{array}{rr}   -1
& 1 \\   0 & 1 \\ \end{array} \right] \\
\left[ \begin{array}{cc}   0 & 0 \\   0 & 0 \\\end{array} \right]&
\left[\begin{array}{cc}  0 & 1 \\  0 & 0 \\ \end{array}\right]
\\ \end{array} \right]$ and consider the idempotent $E=\left[\begin{array}{cc}   \left[ \begin{array}{cc}   0 &
0\\   0 & 1 \\ \end{array} \right] & \left[ \begin{array}{rr}  0
& 0 \\   0 & 0 \\ \end{array} \right] \\
\left[ \begin{array}{cc}   0 & 0 \\   0 & 0 \\\end{array} \right]&
\left[\begin{array}{cc}  0 & 0 \\  0 & 1 \\ \end{array}\right]
\\ \end{array} \right]\in D_2(U_2(R))$. Then $AB=0$ but $BAE\neq 0$. Hence
$D_2(U_2(R))$ is not right $E$-reversible.}
\end{ex}

For a reduced ring $R$, $D_3(R)$ need not be reversible as noted
in \cite[Example 1. 5]{KiL}. We note also the following.

In case $n\geq 3$, the ring $R$ being reduced and $e\in$ Id$(R)$
need not imply $D_n(R)$ being $eI_n$-reversible as illustrated
below.

\begin{ex}{\rm Let $R$ be a reduced ring and $e\in$ Id$(R)$ with $E = eI_3$. Consider
$A = \begin{bmatrix}0&0&0\\0&0&1\\0&0&0\end{bmatrix}$ and $B =
\begin{bmatrix}0&1&0\\0&0&1\\0&0&0\end{bmatrix}\in D_3(R)$. Then
$AB = 0$ and $BAE\neq 0$.}
\end{ex}

However, there are reversible rings and idempotents $E$ in
$D_3(R)$ such that $D_3(R)$ is right $E$-reversible.

\begin{thm}\label{matrix} Let $R$ be a reduced ring and $n$ any positive integer such that $n\geq 3$. Then $D_n(R)$ is right $E_{22}$-reversible ring.
\end{thm}
\begin{proof} Consider $n=3$. Let $A = \begin{bmatrix}a&b&c\\0&a&d\\0&0&a\end{bmatrix}$, $B = \begin{bmatrix}x&y&z\\0&x&t\\0&0&x\end{bmatrix}\in D_3(R)$
with $AB = 0$. Being $R$ reduced gives rise to $BA =
\begin{bmatrix}0&0&*\\0&0&0\\0&0&0\end{bmatrix}$. Hence $BAE = 0$.
\end{proof}

\begin{thm}\label{matrix1} Let $R$ be a reduced ring,
$e\in$ Id$(R)$ and any positive integer $n\geq 3$. Then $V_n(R)$
is right $eI_n$-reversible.
\end{thm}
\begin{proof} Clear by \cite[Theorem 2.5]{KiL}.
\end{proof}

Let $R$ be a ring and $S$ a subring of $R$ with the same identity as that of $R$ and
$$T[R, S] = \{ (r_1, r_2, r_3, \dots , r_n, s, s, s, \dots) : r_i\in R, s\in S; i,n\in \Bbb Z, 1\leq i\leq n\}.$$ Then $T[R, S]$ is
a ring under the componentwise addition and multiplication.

\begin{prop}\label{TRS} Let $R$ be a ring and $S$ a subring of $R$ with the same identity as that of $R$.
Let $e\in$ Id$(S)$ and $E = (e, e, e, \dots)\in$ Id$(T[R, S])$.
Then $R$ is right $e$-reversible if and only if $T[R, S]$ is right
$E = (e, e, e, \dots)$-reversible.
\end{prop}
\begin{proof} For the necessity, assume that $R$ is a right $e$-reversible ring.
Let $A$, $B\in T[R, S]$ with $A = (a_1, a_2, a_3, \dots, a_n, s,
s, s, \dots)$, $B = (b_1, b_2, b_3, \dots, b_m, t, t, t, \dots)\in
T[R, S]$ and $AB = 0$. We may assume that $n\leq m$. Then $a_ib_i
= 0$ where $1\leq i\leq n$, so $b_ia_ie = 0$. If $n + 1\leq i$,
then $sb_i = 0$ and $st = 0$. Hence $b_ise = 0$ and $tse = 0$. It
follows that $BAE = 0$. Similarly, if  $m>n$, then we obtain
$BAE=0$. So $T[R, S]$ is right $E = (e, e, e, \dots)$-reversible.
For the sufficiency, suppose that $T[R, S]$ is right $E = (e, e,
e, \dots)$-reversible. Let $a$, $b\in R$ with $ab = 0$. Set $A =
(a, 0, 0, 0, \dots)$, $B = (b, 0, 0, 0, \dots)\in T[R, S]$. Then
$AB = 0$. By supposition, $BAE = 0$. Hence $bae = 0$. That is, $R$
is right $e$-reversible.
\end{proof}

As an illustration of Proposition \ref{TRS}, we give the following
example.
\begin{ex}{\rm Let $R = M_2(\Bbb Z)$ and $S = U_2(\Bbb Z)$,
$x = \begin{bmatrix}1&0\\1&0\end{bmatrix}$, $y =
\begin{bmatrix}0&0\\1&1\end{bmatrix}\in M_2(\Bbb Z)$ and $e =
\begin{bmatrix}1&1\\0&0\end{bmatrix}$ and $s$, $t\in S$ arbitrary.
Then $xy = 0$. Consider $A = (x, x, x, 0, 0, 0, \dots)$, $B = (y,
y, y, 0, 0, 0, \dots)$ and $E = (e, e, e, e, \dots )$. However,
$xy = 0$ but $yxe\neq 0$. Hence $AB = 0$ but $BAE\neq 0$. By
Example \ref{ikie}, although, $U_2(\Bbb Z)$ is right
$e$-reversible, since $M_2(\Bbb Z)$ is not right $e$-reversible,
$T[M_2(\Bbb Z), U_2(\Bbb Z)]$ is not right $E$-reversible.}
\end{ex}

By a similar discussion in the proof of Proposition \ref{TRS}, we
attain the next result.

\begin{prop} Let $R$ be a ring and $S$ a subring of $R$ with the same identity as that of
$R$. Then the following hold.
\begin{enumerate}
    \item Let $e\in$ Id$(R)$. Then $R$ is right $e$-reversible if and only if $T[R, S]$ is right
$(\underbrace{e, e, \dots, e}_{n~\text{times}},
0,0,\dots)$-reversible for every integer $n\geq 1$.
    \item Let $e_0\in$ Id$(S)$ and $e_1, e_2,\dots,e_n\in$ Id$(R)$. Then $R$
is right $e_i$-reversible for every $i=0,1,\dots,n$ if and only if
$T[R, S]$ is right $(e_1, e_2, \dots,
e_n,e_0,e_0,\dots)$-reversible.
\end{enumerate}
\end{prop}

Let $R$ be a ring and $S$ be a multiplicatively closed subset of
$R$ consisting of central regular elements. Let $S^{-1}R =
\left\{a/s\mid a\in R, s\in S\right\}$ denote the localization of
$R$ at $S$. Note that for any $a/s$, $b/t\in S^{-1}R$, $(a/s)(b/t)
= 0$ if and only if $ab = 0$. Let $a/s\in S^{-1}R$. Then $a/s\in$
Id$(S^{-1}R)$ if and only if  $a^2=as$. We state and prove the
following.
\begin{prop} Let $R$ be a ring and $S$ be a multiplicatively closed subset of $R$ consisting of central regular elements with $e\in$ Id$(R)$. Then $R$ is right $e$-reversible if and only if $S^{-1}R$ is right $(e/1)$-reversible.
\end{prop}
\begin{proof} Assume that $R$ is right $e$-reversible and let $a/s$, $b/t\in S^{-1}R$ with $(a/s)(b/t) = 0$. Then $ab = 0$. Hence $bae = 0$. It follows that $(b/t)(a/s)(e/1) = 0$. Conversely, suppose that $S^{-1}R$ is right $(e/1)$-reversible. Let $a$, $b\in R$ with $ab = 0$. Then $(a/1)(b/1) = 0$. So $(b/1)(a/1)(e/1) = 0$. Hence $bae = 0$.
\end{proof}

\section{Some e-reversible subrings of matrix rings}
In preceding sections we show that full matrix rings $M_n(R)$ and upper triangular matrix rings $U_n(R)$ need not be right(left) $e$-reversible for some idempotent $e$ and for some ring $R$. In this section we investigate the conditions under which right(left) $e$-reversibility properties holds in some subrings of $M_n(R)$.\\

{\bf The rings $H_{(s,t)}(R)$:} Let $R$ be a ring  and  $s$, $t\in
C(R)$ be invertible in $R$. Let\begin{center} $H_{(s,t)}(R) =
\left\{\begin{bmatrix}a&0&0\\c&d&f\\0&0&g
\end{bmatrix}\in M_3(R)\mid a, c, d, f, g\in R, a - d = sc, d - g = tf\right \}$.\end{center}
Then $H_{(s,t)}(R)$ is a subring of $M_3(R)$.

In the following we state and prove the conditions under which
$H_{(s,t)}(R)$ is $e$-reversible. Note that $R$ is a commutative ring if and only if $H_{(s,t)}(R)$ is a commutative ring.
\begin{lem}\label{hu} Let $R$ be a ring with $e\in$ Id$(R)$ and $E = eI_3\in H_{(s,t)}(R)$.
Then $R$ is  a right $e$-reversible ring if and only if
$H_{(s,t)}(R)$ is a right $E$-reversible ring.
\end{lem}
\begin{proof} For the necessity, assume that $R$ is a right $e$-reversible ring.
Let $A = \begin{bmatrix}a&0&0\\c&d&f\\0&0&g\end{bmatrix}$, $B =
\begin{bmatrix}x&0&0\\y&z&u\\0&0&v\end{bmatrix}\in H_{(s,t)}(R)$
with $AB = 0$. Then $ax = 0$, $dz = 0$, $gv = 0$, $cx + dy = 0$
and $du + fv = 0$. Then $xae = 0$, $zde = 0$, $vge = 0$. First we
show $(ya + zc)e = 0$ and $(zf + ug)e = 0$ to reach $BAE = 0$. We
use $a - d = sc$, $d - g = tf$, $x - z = sy$ and $z - v = tu$ in
the sequel without reference. By using these equalities, we have
$ya + zc = s^{-1}(x - z)a + zc = s^{-1}xa - s^{-1}za + s^{-1}(szc)
= s^{-1}xa - s^{-1}z(a - sc) = s^{-1}xa - s^{-1}zd = s^{-1}(xa -
zd)$. Multiplying the latter equalities on the right by $e$ yields
$(ya + zc)e = s^{-1}(xa - zd)e = s^{-1}xae - s^{-1}zde = 0$ since
$xae = 0$ and $zde = 0$. Similarly, $zf + ug = t^{-1}(ztf) +
t^{-1}(z - v)g = t^{-1}z(tf + g) - t^{-1}vg = t^{-1}zd -
t^{-1}vg$. Multiplying the latter equalities on the right by $e$
we get $(zf + ug)e =  t^{-1}zde - t^{-1}vge = 0$ since $zde = 0$
and $vge = 0$. It follows that $BAE = 0$. Therefore $H_{(s,t)}(R)$
is right $E$-reversible. For the sufficiency, suppose that
$H_{(s,t)}(R)$ is a right $E$-reversible ring. Let $a$, $b\in R$
with $ab = 0$. Let $A = aI_3$, $B = bI_3$. Then $AB = 0$. By
supposition, $BAE = 0$. It gives us $bae = 0$. Hence $R$ is right
$e$-reversible.
\end{proof}
We use the proof of Lemma \ref{hu} to complete the proof of the
following result.
\begin{thm} Let $R$ be a ring with $e\in$ Id$(R)$ and consider the ring $H_{(s,t)}(R)$.
\begin{enumerate}
\item[{\rm (1)}] Let $s = t = 1$. Then Id$(H_{(1,1)}(R)) = \{eI_3,
E_1 = ee_{11} + ee_{21}, E_2 = -ee_{23} + ee_{33}, E_3 = E_1 +
E_2, E_4 = -ee_{21} + ee_{22} + ee_{23}, E_5 = ee_{11} + ee_{22} +
ee_{23}, E_6 = -ee_{21} + ee_{22} + ee_{33}\}$ is a collection of
idempotents in $H_{(1,1)}(R)$. Moreover, $R$ is right
$e$-reversible if and only if $H_{(1,1)}(R)$ is right
$E$-reversible for each $E\in$ Id$(H_{(1,1))}(R)$.

\item[{\rm (2)}] Let $s\neq 1$ and $t = 1$. Then Id($H_{(s,1)}(R)) =
\{eI_3, F_1 = ee_{11} + ee_{22} + ee_{23}, F_2 = -s^{-1}ee_{21} +
ee_{22} + ee_{33}\}$ is a collection of idempotents in
$H_{(s,1)}(R)$. Moreover, $R$ is right $e$-reversible if and only
if $H_{(s,1)}(R)$ is right $E$-reversible for each $E\in$
Id$(H_{(s,1)}(R))$.

\item[{\rm (3)}] Let $s = 1$ and $t\neq 1$. Then Id($H_{(1,t)}(R)) =
\{eI_3, G_1 = ee_{11} + ee_{22} + t^{-1}ee_{23}, G_2 = -ee_{21} +
ee_{22} + ee_{33}\}$ is a collection of idempotents in
$H_{(1,t)}(R)$. Moreover, $R$ is right $e$-reversible if and only
if $H_{(1,t)}(R)$ is right $E$-reversible for each $E\in$
Id$(H_{(1,t)}(R))$.

\item[{\rm (4)}] Let $s \neq 1$ and $t\neq 1$. Then Id($H_{(s,t)}(R))
= \{eI_3, H_1 = F_2, H_2 = G_1, H_3 = -s^{-1}ee_{21} + ee_{22} +
t^{-1}ee_{23}\}$ is a collection of idempotents in $H_{(s,t)}(R)$.
Moreover, $R$ is right $e$-reversible if and only if
$H_{(s,t)}(R)$ is right $E$-reversible for each $E\in$
Id$(H_{(s,t)}(R))$.
\end{enumerate}
\end{thm}
\begin{proof} For the sufficiency in each case, since $H_{(s,t)}(R)$ is right
$eI_3$-reversible, $R$ is right $e$-reversible by Lemma \ref{hu}.
For the necessity, let $R$ be right $e$-reversible.\\
 (1) Assume that $s = t = 1$.
Let $A = \begin{bmatrix}a&0&0\\c&d&f\\0&0&g\end{bmatrix}$, $B =
\begin{bmatrix}x&0&0\\y&z&u\\0&0&v\end{bmatrix}\in H_{(1,1)}(R)$
with $AB = 0$. Then $ax = 0$, $dz = 0$, $gv = 0$, $cx + dy = 0$
and $du + fv = 0$. Since $R$ is right $e$-reversible, we have $xae
= 0$, $zde = 0$, $vge = 0$, $(ya + zc)e = 0$ and $(zf + ug)e = 0$.
Consider the following items:
\begin{itemize}
    \item $BAE_1 = 0$ since $zde = 0$ and
$(ya + zc)e = 0$.
    \item $BAE_2 = 0$ since $-zde + (zf + ug)e = 0$ and
$vge = 0$.
    \item $E_3 = E_1 + E_2$ yields $BAE_3 = 0$.
    \item $BAE_4 = 0$
since $zde = 0$.
    \item $BAE_5 =
\begin{bmatrix}xae&0&0\\(ya+zc)e&zde&zde\\0&0&0\end{bmatrix} = 0$
due to $xae = 0$, $zde = 0$ and $xae - zde = (ya + zc)e = 0$.
    \item $BAE_6 =
\begin{bmatrix}0&0&0\\-zde&zde&(zf+ug)e\\0&0&vge\end{bmatrix} = 0$
due to $zde = 0$, $vge = 0$ and $(zf+ug)e = 0$.
\end{itemize}
The rest is proved similarly.
\end{proof}

{\bf Generalized matrix rings:} Let $R$ be a ring and $s$ a
central element of $R$. Then  $\begin{bmatrix}
R&R\\R&R\end{bmatrix}$ becomes a ring denoted by $K_s(R)$ with
addition defined componentwise and multiplication defined in
\cite{KT} by
$$\begin{bmatrix} a_1&x_1\\y_1&b_1\end{bmatrix}\begin{bmatrix}
a_2&x_2\\y_2&b_2\end{bmatrix} = \begin{bmatrix}a_1
a_2+sx_1y_2&a_1x_2+x_1b_2\\y_1a_2+b_1y_2&sy_1x_2+b_1b_2\end{bmatrix}.$$
In \cite{KT}, $K_s(R)$ is called a {\it generalized matrix ring
over $R$}. There are rings $R$ such that $K_0(R)$ need not be
right $E$-reversible for each $0\neq E\in$ Id$(K_0(R))$ as shown
below.

\begin{prop}\label{inv} The ring $K_0(\Bbb Z)$ is not $E$-reversible for each $0\neq E\in$ Id$(K_0(\Bbb Z))$.
\end{prop}
\begin{proof} It is easily checked that Id$(K_0(\Bbb Z))$ consists of $0$, $I_2$, all matrices of the form $\begin{bmatrix}
1&x\\y&0\end{bmatrix}$ and all matrices of the form
$\begin{bmatrix} 0&x\\y&1\end{bmatrix}$ where $x,y\in \Bbb Z$.
Consider the following cases:\\
\begin{enumerate}
    \item Assume that $E=I_2$. For $A = \begin{bmatrix}
1&0\\1&0\end{bmatrix}$, $B = \begin{bmatrix}
0&0\\1&1\end{bmatrix}\in K_0(\Bbb Z)$, we have $AB=0$ but $BAE\neq
0$.
    \item Assume that $E$ is of the form $\begin{bmatrix}
1&x\\y&0\end{bmatrix}$ where $x,y\in \Bbb Z$. For $A =
\begin{bmatrix} 1&1\\1&0\end{bmatrix}$, $B = \begin{bmatrix}
0&0\\1&0\end{bmatrix}\in K_0(\Bbb Z)$, we have $AB=0$ but $BAE\neq
0$.
    \item Assume that $E$ is of the form $\begin{bmatrix} 0&x\\y&1\end{bmatrix}$  where $x,y\in \Bbb Z$. For $A =
\begin{bmatrix} 0&1\\1&1\end{bmatrix}$, $B = \begin{bmatrix}
1&0\\-1&0\end{bmatrix}\in K_0(\Bbb Z)$, we have $AB=0$ but
$BAE\neq 0$.
\end{enumerate}
\end{proof}

\end{document}